\documentclass[a4paper,10pt]{amsart}
\usepackage[arrow,matrix]{xy}
\usepackage{amsmath,amssymb,amscd,bbm,amsthm,mathrsfs,dsfont}
\theoremstyle{plain}
\newtheorem{theorem}{Theorem}[section]
\newtheorem{corollary}[theorem]{Corollary}
\newtheorem{lemma}[theorem]{Lemma}
\newtheorem{prop}[theorem]{Proposition}
\theoremstyle{definition}
\newtheorem{defn}[theorem]{Definition}
\newtheorem{exa}[theorem]{Example}

\theoremstyle{remark}

\newtheorem{remark}{Remark}[section]

\DeclareMathOperator{\Ext}{Ext} \DeclareMathOperator{\Hom}{Hom}

\begin{document}
\title[Koszulity of algebras with non-pure resolutions]
{\bf Koszulity of algebras with non-pure resolutions}

\author{Di-Ming Lu}
\address{(Lu) Department of Mathematics, Zhejiang University, Hangzhou 310027, China}
\email{dmlu@zju.edu.cn}
\author{Jun-Ru Si}
\address{(S\,i) Department of Mathematics, Zhejiang University, Hangzhou 310027, China}
\email{sijunru@126.com}

\keywords{$N$-Koszul algebra, piecewise-Koszul algebra, AS-regular
algebra}
\thanks {The work was supported by the NSFC (Grant No. 10571152) and partially by
the NSF of Zhejiang Province of China (Grant No. J20080154)}
\subjclass[2000]{16E05, 16E40, 16S37, 16W50.}
\date{}
\maketitle

\begin{abstract}
We discuss certain homological properties of graded algebras whose
trivial modules admit non-pure resolutions. Such algebras include
both of Artin-Schelter regular algebras of types (\emph{12221}) and
(\emph{13431}). Under certain conditions, a module with non-pure
resolution is decomposed to form an extension by two modules with
pure resolutions.
\end{abstract}

\vskip7mm
\section*{Introduction}

Consider a (connected) graded algebra $A=k\oplus A_1\oplus A_2\oplus
\cdots$. Let $V\subset A$ be a minimal graded vector space which
generates $A$, and $R\subset T\langle V \rangle$ the minimal graded
vector space which generates the relations of $A$. Then $A\cong
T\langle V\rangle/(R)$ where $(R)$ is the ideal generated by $R$,
and the start of a minimal graded free resolution of the trivial
left $A$-module $_Ak$ is
$$
\cdots \to A\otimes R\to A\otimes V \to A\to k\to 0.
$$
More information is lurking in the subsequent terms of the
resolution, and one may expect to understand for what they are
standing by examining certain simple resolutions.

Perhaps the linear resolutions are the most suitable expecting. A
well-known such example is {\it Koszul algebras\/}, introduced by
Priddy in \cite{P}. There are some generalizations on Koszul
algebras. For example, {\it $N$-Koszul algebras\/}, motivated by the
cubic Artin-Schelter regular algebras (AS-regular algebras, for
short), were introduced by Berger in \cite{Be}, recent developing on
Calabi-Yau algebras offered a new class of $N$-Koszul algebras
(\cite{BT, Bo}). {\it Almost Koszul algebras\/}, grew out of finding
periodic resolutions for the trivial extension algebras of the path
algebras of Dynkin quivers, were defined in \cite{BBK}. Several
progresses have been given since then, such as {\it stacked monomial
algebras\/} in \cite{GS} where the authors pointed out that all
monomial algebras with pure resolutions are stacked monomial
algebras, and {\it piecewise-Koszul algebras\/} in \cite{LHL}. Note
that all the objects above bear with {\it pure\/} resolutions,
purity becomes a powerful homological tool. Our starting idea for
writing this article is trying to understand the resolutions in some
{\it non-pure\/} case.

There is an alternate generalization of the notion by using the
duality of Koszul algebras. Denote $E(A)$ the Ext-algebra of $A$,
Cassidy and Shelton recently introduced the notion of {\it
$\mathcal{K}_2$ algebra \/}in \cite{CS}, for a connected graded
algebra, by defining its Ext-algebra $E(A)=(E^1(A), E^2(A))$ as an
algebra. Both Koszul algebras and $N$-Koszul algebras fall into the
class of $\mathcal{K}_2$ algebras. On the other hand, a
piecewise-Koszul algebra, determined by a pair of parameters $(p,
N)$ ($2\le p\le N$), is $\mathcal{K}_2$ if and only if the period
$p=2$ or $p=N$. Another remarkable example of $\mathcal{K}_2$
algebras is the AS-regular algebras of type (\emph{13431}), and yet
the AS-regular algebras of type (\emph{12221}) are not
$\mathcal{K}_2$.

The AS-regular algebras have been studied in many recent papers. In
particular, four families of non-Koszul regular algebras were
constructed in \cite{LPWZ} by using the theory of
$A_\infty$-algebras. In that paper, the AS-regular algebras of
global dimension 4 generated in degree 1 are classified into the
following three types according to the number of generators:

\begin{itemize}
{\small
\item [(\emph{12221})]\quad $0\rightarrow A(-7)\rightarrow A(-6)^{\oplus2}\rightarrow
A(-4)\oplus A(-3) \rightarrow A(-1)^{\oplus2}\rightarrow A\rightarrow  k\rightarrow 0$,
\item [(\emph{13431})]\quad $0\rightarrow
A(-5)\rightarrow A(-4)^{\oplus3}\rightarrow A(-3)^{\oplus2}\oplus
A(-2)^{\oplus2}\rightarrow A(-1)^{\oplus3}\rightarrow A\rightarrow
k\rightarrow 0$,
\item [(\emph{14641})]\quad $0\rightarrow A(-4)\rightarrow
A(-3)^{\oplus4}\rightarrow A(-2)^{\oplus6}\rightarrow
A(-1)^{\oplus4}\rightarrow A\rightarrow k\rightarrow 0$.
}
\end{itemize}

The algebras of type (\emph{14641}) are Koszul. Our goal in this
article is to introduce a new class of Koszul-type algebras, called
{\it bi-Koszul algebras\/}, which includes both AS-regular algebras
of types (\emph{12221}) and (\emph{13431}) as its objects. Such
algebras are defined as non-pure analogue of the piecewise-Koszul
algebras (see Definition \ref{bk}).

\medskip

We can connect the theory of the bi-Koszul algebras and their Koszul
duality with the following, to be proved in Section \ref{sec}:
\begin{theorem}\label{thm0} $A$ is a bi-Koszul algebra if and only if its Koszul
dual $E(A)$ begins with $E^1(A)= E^1_1(A)$, $E^2(A)=
E^2_d(A)\oplus E^2_{d+1}(A)$, $E^3(A)= E^3_{2d}(A)$, and
$E^{3n}(A)= (E^3(A))^n$, $E^{3n+1}(A)= E^1(A)E^{3n}(A)$,
$E^{3n+2}(A)\cong E^2(A)E^{3n}(A)\oplus \mathsf{E}^2_{2nd+d+1}
(J\Omega^{3n}(k))$ as $E^0(A)$-modules.
\end{theorem}

We consider a strongly bi-Koszul algebra $A$ with its Ext-algebra
$E(A)$ being generated by $E^i(A)$ ($i=0, 1, 2, 3$), and give a
construction of such algebras.

\medskip
The notion of (strongly) bi-Koszul modules is introduced in Section
\ref{sect}. One of main homological properties of bi-Koszul modules
is (see the theorems \ref{thmm1} and \ref{thm6}):
\begin{theorem}
Let $A$ be a bi-Koszul algebra and $M\in Gr_0(A)$. We have:
\begin{enumerate}
\item If $\mathsf{E}(M)$ is generated by $\mathsf{E}^0(M)$, then $M$ is a bi-Koszul module.
\item If $M$ is a strongly bi-Koszul module, then $\mathsf{E}(M)$ is
generated by $\mathsf{E}^0(M)$.
\end{enumerate}
\end{theorem}

The resolutions of bi-Koszul modules are non-pure, which results an
obstruction in the study of homological properties. A natural
question is: what non-pure resolutions can be decomposed into pure
resolutions. We discuss the question for a kind of decomposable
modules in Section \ref{secti}.

\begin{theorem} Let $A$ be a connected noetherian graded algebra, and
$M$ a decomposable bi-Koszul module (see Section \ref{secti}). Then
the resolution $\mathcal{Q}$ of $M$ can be decomposed into a direct
sum of two pure resolutions. Moreover, there exists a short exact
sequence $0\rightarrow M'\rightarrow M\rightarrow M''\rightarrow0$
such that both $M'$ and $ M''$ are $ \delta$-Koszul modules.
\end{theorem}

\vskip5mm
\section{Preliminaries}

Throughout this article, $\mathds{F}$ is a fixed field. We always
assume that $A=k\oplus A_1\oplus A_2\oplus\cdots$ is an
associated graded $\mathds{F}$-algebra, where $k$ is a semisimple
Artin algebra over $\mathds{F}$, $A_1$ is a finitely generated
$\mathds{F}$-module, and $A$ is generated in degree 0 and 1. The
graded Jacobson radical of $A$, denoted by $J$, is $J=A_1\oplus
A_2\oplus\cdots$. Let $Gr(A)$ denote the category of graded left
$A$-modules, and $gr(A)$ its full subcategory of finitely generated
$A$-modules. The morphisms in these categories, denoted by
$\Hom_{Gr(A)}(M, N)$, are graded $A$-module maps of degree zero. We
denote $Gr_0(A)$ and $gr_0(A)$ the full subcategories of $Gr(A)$ and
$gr(A)$ whose objects are generated in degree $0$, respectively.

Let $M\in Gr(A)$, we denote the $n^\mathrm{th}$ {\it shift\/} of $M$
by $M(n)$ where $M(n)_j=M_{j-n}$. In particular, for a graded
algebra $A$, we use the notation established in \cite{CS}, setting
$$A(m_1, \cdots, m_s):=A(m_1)\oplus \cdots \oplus A(m_s).$$

We write $\underline{\Ext}^*_A$ the derived functor of the graded
$\underline{\Hom}^*_A$ functor
$$
\underline{\Hom}^*_A(M, N):=\bigoplus_n \Hom_{Gr(A)}(M, N(n)),
$$
and denote
$$
E(A): = \underline{\Ext}_A^*(k, k),\quad  \quad
\mathsf{E}(M): = \underline{\Ext}_A^*(M, k),
$$
the {\it Koszul
dual\/} (or Ext-algebra) of the algebra $A$ and the {\it Koszul
dual\/} of the module $M\in Gr(A)$, respectively. $\mathsf{E}(M)$ is
a left $E(A)$-module by the Yoneda product as a bigraded space with
the $(i, j)^\textmd{th}$ component
$\mathsf{E}^i_j(M):=\underline{\Ext}_A^i(M, k)_j$. Similarly, $E(A)$
is a bigraded algebra with $E^0(A)=k$.

All modules bounded below in $Gr(A)$ have minimal graded projective
resolutions. Thus, we fix
$$
\mathcal{P}:\quad \cdots \rightarrow P_n \stackrel{d_n}
\rightarrow \cdots \rightarrow P_2\stackrel{d_2}\rightarrow
P_1\stackrel{d_1}\rightarrow P_0\stackrel{d_0}\rightarrow
k\rightarrow 0
$$
a minimal graded projective resolution of $k$ over $A$, where $\ker
d_n\subseteq JP_n$ for all $n\geq 0$. And assume
$$
\mathcal{Q}: \quad \cdots \rightarrow Q_n \stackrel{\partial_n}
\rightarrow \cdots \rightarrow Q_2\stackrel{\partial_2}\rightarrow
Q_1 \stackrel{\partial_1} \rightarrow
Q_0\stackrel{\partial_0}\rightarrow M\rightarrow 0
$$
is a minimal graded projective resolution of $M \in Gr(A)$, where
$\ker\partial_n\subseteq JQ_n$ for all $n\geq 0$. Denote
$\Omega^n(M)=\ker \partial_{n-1}$ the $n^\mathrm{th}$ {\it
syzygy\/} of $M$.

We say a graded module $N$ {\it supported\/} in $S$ if $N_i=0$ for
all $i\notin S$. Now, for each $n\geq 0$, $Q_n$ is graded by
internal (or second) degrees, we write $Q_n=\oplus_i Q_{n, i}$. Let
$Q_n$ be supported in $\{j \ |\ j \geq i\}$, then the following are
clear:
\begin{enumerate}
\item $\textsf{E}^n(M) = \underline{\Hom}_A(Q_n, k)
=\bigoplus_{j\geq i}\Hom_{Gr(A)}(Q_n, k(j))$;
\item there exists an integer $s\ (\geq i)$ such that for any $j\geq s,
\; Q_{n, j}=A_{j-s}Q_{n, s}$ if and only if $\Hom_{Gr(A)}(Q_n, k(j))=0$.
\end{enumerate}

From these facts, we have

\begin{lemma}\label{le}
$Q_n$ is generated in degrees $i_1, i_2, \cdots, i_l$
if and only if $\mathsf{E}^n(M)$ is supported in
$\{i_1, i_2, \cdots, i_l\}$. \qed
\end{lemma}

Since $\mathsf{E}(M)$ is bigraded for any $M \in Gr(A)$, we say
$\mathsf{E}^n(M)$ is {\it pure\/} in the sense that it is
supported in a single second degree for any $n\geq 0$. A {\it pure
resolution\/} means every projective module in a minimal projective
resolution is generated in one degree; if not, we call it a {\it
non-pure resolution}. The following result is well-known for pure
parts in the resolution.
\begin{lemma}[\cite{GMMZ} Proposition 3.6]\label{gmmz}
Suppose that $P_n$ is generated in single degree $d_n$ for $n=i,
j,\, i+j$ satisfying $d_i+ d_j= d_{i+j}$. Then
$$
E^{i+j}(A)= E^i(A)E^j(A)=E^j(A)E^i(A).
$$
\end{lemma}

The Koszulity of algebras with pure resolutions has been studied in
many papers (see, for example, \cite{Be, BBK, GM, GMMZ, GS, LHL,
P}). We refer to \cite{BBK} for the notion of the almost Koszul
algebras, to \cite{GM} for the $\delta$-Koszul algebras, to
\cite{LHL} for the piecewise-Koszul algebras. As mentioned in the
introductory section, the latter is related to a pair of integers
$(p, N)$, the parameter $p$ shows the periodicity of generators'
degree distribution in the resolution, and the other one is related
to the jumping degree. The definition agrees with the classical
Koszul algebra when the period equals to the jumping degree, and
goes back to the $N$-Koszul algebra when the period $p=2$. One may
imply, on the other hand, that an almost Koszul algebra must be a
piecewise-Koszul algebra, moreover, the two concepts become
consistent if and only if $A=A_0\oplus A_1\oplus\cdots \oplus A_q$
and $N\ge p+q-1$. A criterion theorem for a graded algebra $A$ to be
piecewise-Koszul is that $E(A)$ is generated by $E^0(A), E^1(A),
E^p(A)$, and $E^p(A)=E^p_N(A)$.

By comparison, within our knowledge, we have less clue, except the
recent paper \cite{CS}, for the study of Koszulity of algebras whose
trivial modules admitting non-pure resolutions.

\vskip5mm
\section{Bi-Koszul Algebras}\label{sec}

In this section, we give the definition of (strongly) bi-Koszul
algebras and discuss their homological properties.

To meet the non-pure situation, we first introduce a {\it resolution
map}
$$
\varDelta: \mathbb{N}\longrightarrow \mathbb{N}\times \mathbb{N}
$$
by
$$
\varDelta(n)=\left\{\begin{array}{llll}
\frac{n}{3}(2d, 2d), & \mbox{if $n \equiv 0 (\textrm{mod} 3)$,}\\
\frac{n-1}{3}(2d, 2d)+(1,1), & \mbox{if $n \equiv 1 (\textrm{mod} 3)$,}\\
\frac{n-2}{3}(2d, 2d)+(d, d+1), & \mbox{if $n \equiv 2 (\textrm{mod}
3)$,}\end{array} \right.
$$
where $d\geq 2$ is a fixed integer, and $\mathbb{N}\times \mathbb{N}$
is a $\mathbb{Z}$-module with natural operations.

It is clear that the resolution map has a period of three; that is,
for all $n\geq 0$ ,
$$
\varDelta(n+3)=\varDelta(n)+\varDelta(3).
$$

For simplification, we use $\varDelta(n)$ to express both of its
image $(x, y)$ and of the set $\{x, y\}$, so $\varDelta(0)=\{ 0 \}$,
$\varDelta(1)=\{ 1 \}$ and $\varDelta(2)=\{d,\ d+1\}$.

\begin{defn}\label{bk}
A graded algebra $A=k\oplus A_1\oplus A_2\oplus\cdots$ is called a
{\it{bi-Koszul algebra\/}} if the trivial left $A$-module $k$ admits
a minimal graded projective resolution
$$
\mathcal{P}: \quad \cdots \rightarrow P_n\rightarrow \cdots \rightarrow
P_1\rightarrow P_0\rightarrow k\rightarrow 0,
$$
such that each $P_n$ is generated in degrees $\varDelta(n)$ for all
$n\geq 0$.
\end{defn}

Here are some examples of the bi-Koszul algebras.
\begin{exa}
The AS-regular algebras of global dimension 4 of type
(\emph{13431}) and (\emph{12221}) are the bi-Koszul algebras with
taking $d=2$ and $d=3$, respectively.
\end{exa}
\begin{exa}
Let $\Gamma$ be the following quiver:
$$
\bullet \stackrel{\alpha}\rightarrow \bullet\stackrel{\beta}\rightarrow
\bullet\stackrel{\gamma}\rightarrow\bullet\stackrel{\delta}\rightarrow
\bullet\stackrel{\pi}\rightarrow\bullet
$$
Set $A=\mathds{F}\Gamma/R$, where $R$ is the ideal generated by the
relations $\alpha\beta,\;\beta\gamma\delta$, and $\delta\pi$. One
can easily check that $A$ is a bi-Koszul algebra.
\end{exa}

One of effective approaches for studying the algebra $A$ is to
examine by its Ext-algebra $E(A)$. The resolution $\mathcal{P}$ of a
bi-Koszul algebra $A$ contains pure and non-pure parts. It is easy
to deal with the pure part by noting Lemma \ref{gmmz} and the
periodicity of the resolution map $\varDelta$.

\begin{prop}\label{propp}
Let $A$ be a bi-Koszul algebra. Then
$$
E^{3n}(A)= \underbrace{E^3(A)E^3(A)\cdots E^3(A)}_n
$$
and
$$
E^{3n+1}(A)= E^1(A)E^{3n}(A)=E^{3n}(A)E^1(A).
$$
\end{prop}
\begin{proof} It is clear.
\end{proof}

Suppose $A$ is a bi-Koszul algebra, we write
$$
E^{[P]}(A):=\big(\bigoplus_{n\geq 0}E^{3n}(A) \big)\bigoplus
\big(\bigoplus_{n\geq 0}E^{3n+1}(A)\big),
$$
and
$$
E^{[N]}(A):=\bigoplus_{n\geq 0}E^{3n+2}(A)
$$
respectively, then $E(A)$ is decomposed into pure and non-pure parts
as $E(A)= E^{[P]}(A) \oplus E^{[N]}(A)$.
\begin{corollary}
Assume $d>2$. We have
\begin{enumerate}
\item $E^{[P]}(A)$ is a
subalgebra of $E(A)$ which is generated by $E^0(A)$, $E^1(A)$
and $E^3(A)$.
\item $E^{[N]}(A)$ is a module over $E^{[P]}(A)$.
\end{enumerate}
\end{corollary}
\begin{proof} By noting the second degrees, one can easily check that
$E^{3s+1}(A)E^{3t+2}(A)=E^{3s+2}(A)E^{3t+1}(A)=
E^{3s+1}(A)E^{3t+1}(A)= 0$ for any integers $s, t\geq 0$. The result
then follows from Lemma \ref{le} and Proposition \ref{propp}.
\end{proof}

For non-pure part of the resolution, there is an obstruction because
of its non-purity. We quote two lemmas firstly.

\begin{lemma}[\cite{GMMZ} Lemma 3.2]
Let $M\in Gr(A)$ supported in $\{j\ |\ j\geq0\}$ with the minimal
graded projective resolution $\mathcal{Q}$. Let $n\geq 1$, assume
that $P_n$ is supported in $\{j\mid j\geq s\}$. Then $Q_n$ is
supported in $\{j\ |\ j\geq s\}$. \qed
\end{lemma}

\begin{lemma}\label{le2}
Let $M \in Gr_0(A)$. For the natural induced map
$$
\mathsf{E}^n(M/JM)\stackrel{f}\rightarrow \mathsf{E}^n(M)\stackrel{g}
\rightarrow \mathsf{E}^n(JM),
$$
we have $\mathrm{Im}(f)= E^n(A)\mathsf{E}^0(M)$, the Yoneda
product of $E^n(A)$ and $\mathsf{E}^0(M)$.
\end{lemma}
\begin{proof} The result is followed from the commutative diagram below:
$$
\begin{array}{ccccccccccc} & E^n(A)\otimes \mathsf{E}^0(M/JM) & \twoheadrightarrow
&E^n(A)\mathsf{E}^0(M/JM) & \stackrel{=}\rightarrow &
\mathsf{E}^n(M/JM)& \\ & \downarrow \cong && \downarrow &&\downarrow f & \\
& E^n(A)\otimes \mathsf{E}^0(M)&\twoheadrightarrow&
E^n(A)\mathsf{E}^0(M) & \hookrightarrow& \mathsf{E}^n(M ) &
\end{array}
$$
which can be found in \cite{GMMZ}.
\end{proof}

\begin{prop}\label{prop3}
Let $A$ be a bi-Koszul algebra. Let $M \in Gr_0(A)$ with the minimal
graded projective resolution $\mathcal{Q}$. Assume that $Q_n$ is
generated in degrees $\varDelta(n)$. Then there exist $k$-module
isomorphisms, for all $n\geq 0$
$$
\mathsf{E}^{3n+2}(M)\cong E^{3n+2}(A)\mathsf{E}^0(M)\oplus
\mathsf{E}^{3n+2}_{2nd+d+1}(JM).
$$
\end{prop}
\begin{proof}
Since $JM(-1)\in Gr_0(A) $, $\mathsf{E}^n(JM)$ is supported in
$\{j\ |\ j\geq \mathrm{min} \varDelta(n)+1\}$. Consider the graded
short exact sequence
$$
0\rightarrow JM\rightarrow M\rightarrow M/JM\rightarrow 0.
$$
Applying  $\underline{\Ext}_A^*(-, k)$, we obtain a long
exact sequence
$$
\begin{array}{r}0\rightarrow \mathsf{E}^0(M/JM)\rightarrow \mathsf{E}^0(M)
\rightarrow \mathsf{E}^0(JM)\rightarrow
\mathsf{E}^1(M/JM)\rightarrow \mathsf{E}^1(M)\rightarrow
\mathsf{E}^1(JM)\rightarrow\\\rightarrow
\mathsf{E}^2(M/JM)\rightarrow \mathsf{E}^2(M)\rightarrow
\mathsf{E}^2(JM)\rightarrow \mathsf{E}^3(M/JM)\rightarrow
\mathsf{E}^3(M)\rightarrow \mathsf{E}^3(JM)\rightarrow
\\\rightarrow \mathsf{E}^4(M/JM)\rightarrow
\mathsf{E}^4(M)\rightarrow \mathsf{E}^4(JM)\rightarrow
\mathsf{E}^5(M/JM) \rightarrow \mathsf{E}^5(M)\rightarrow
\mathsf{E}^5(JM)\rightarrow
\\\rightarrow \mathsf{E}^6(M/JM)\rightarrow
\mathsf{E}^6(M)\rightarrow
\mathsf{E}^6(JM)\rightarrow\cdots.\hskip49.5mm\end{array}
$$
Comparing the second degrees of all modules in the long exact
sequence above, we have a series of exact sequences:
$$
0\rightarrow \mathsf{E}^0(M/JM)\rightarrow
\mathsf{E}^0(M)\rightarrow 0,
$$
$$
0\rightarrow \mathsf{E}^0_1(JM)\rightarrow
\mathsf{E}^1_1(M/JM)\rightarrow
\mathsf{E}^1_1(M)\rightarrow 0,
$$
$$
0\rightarrow \mathsf{E}^1_d(JM)\rightarrow
\mathsf{E}^2_d(M/JM)\rightarrow
\mathsf{E}^2_d(M)\rightarrow 0,
$$
$$
0\rightarrow \mathsf{E}^1_{d+1}(JM)\rightarrow
\mathsf{E}^2_{d+1}(M/JM)\rightarrow
\mathsf{E}^2_{d+1}(M)\rightarrow
\mathsf{E}^2_{d+1}(JM)\rightarrow 0,
$$
$$0\rightarrow \mathsf{E}^2_{2d}(JM)\rightarrow
\mathsf{E}^3_{2d}(M/JM)\rightarrow
\mathsf{E}^3_{2d}(M)\rightarrow 0,
$$
$$
0\rightarrow \mathsf{E}^3_{2d+1}(JM)\rightarrow
\mathsf{E}^4_{2d+1}(M/JM)\rightarrow
\mathsf{E}^4_{2d+1}(M)\rightarrow 0,
$$
$$
0\rightarrow \mathsf{E}^4_{3d}(JM)\rightarrow
\mathsf{E}^5_{3d}(M/JM)\rightarrow
\mathsf{E}^5_{3d}(M)\rightarrow 0,
$$
$$
0\rightarrow \mathsf{E}^4_{3d+1}(JM)\rightarrow
\mathsf{E}^5_{3d+1}(M/JM)\rightarrow
\mathsf{E}^5_{3d+1}(M)\rightarrow
\mathsf{E}^5_{3d+1}(JM)\rightarrow 0,
$$
$$
0\rightarrow \mathsf{E}^5_{4d}(JM)\rightarrow
\mathsf{E}^6_{4d}(M/JM)\rightarrow
\mathsf{E}^6_{4d}(M)\rightarrow 0,
$$
$$
\cdots\cdots.
$$

Hence, by Lemma \ref{le2},  there are exact sequences:
$$
0\rightarrow E^{3n+2}_{2nd+d}(A)\mathsf{E}^0(M)\rightarrow
\mathsf{E}^{3n+2}_{2nd+d}(M )\rightarrow 0,
$$
and
$$
0\rightarrow E^{3n+2}_{2nd+d+1}(A)\mathsf{E}^0(M)\rightarrow
\mathsf{E}^{3n+2}_{2nd+d+1}(M )\rightarrow
\mathsf{E}^{3n+2}_{2nd+d+1}(JM )\rightarrow 0,
$$
for any $n\geq 0$. Note that the second short exact sequence is split
as $k$-modules, and both $E^{3n+2}(A)$ and $\mathsf{E}^{3n+2}(M)$
are supported only in $\varDelta(3n+2)$, so we complete the proof.
\end{proof}

We may now prove Theorem \ref{thm0}, which we restate:
\begin{theorem} \label{thm1}
The following statements are equivalent:
\begin{enumerate}
\item $A$ is a bi-Koszul algebra;
\item $E(A)$ begins with $E^1(A)=
E^1_1(A)$, $E^2(A)= E^2_d(A)\oplus E^2_{d+1}(A)$,
$E^3(A)= E^3_{2d}(A)$, and for each $n\geq 1$,
\begin{enumerate}
\item $E^{3n}(A)= \overbrace{E^3(A)E^3(A)\cdots E^3(A)}^n$,
\item $E^{3n+1}(A)= E^1(A)E^{3n}(A)$,
\item $E^{3n+2}(A)\cong E^2(A)E^{3n}(A)\oplus \mathsf{E}^2_{2nd+d+1}
(J\Omega^{3n}(k))$ as $k$-modules.
\end{enumerate}
\end{enumerate}
\end{theorem}
\begin{proof}
Suppose we have (2), it is easy to see that $P_1$ is
$\varDelta(1)$-generated, $P_2$ is $\varDelta(2)$-generated and
$P_3$ is $\varDelta(3)$-generated. For all $n\geq 1$, $P_{3n}$ and
$P_{3n+1}$ are generated in degree $\varDelta(3n)$ and
$\varDelta(3n+1)$ from the conditions (a) and (b), respectively. On
the other hand, by the condition (c) and Proposition \ref{prop3}, we
get
$$
E^{3n+2}(A)=\mathsf{E}^2(\Omega^{3n}(k))\cong
E^2(A)\mathsf{E}^0(\Omega^{3n}(k))\oplus
\mathsf{E}^2_{2nd+d+1}(J\Omega^{3n}(k)).
$$
$E^2(A)$ is supported in $\{d, d+1\}$, and $\Omega^{3n}(k)$ is
$2nd$-generated, so $E^{3n+2}(A)$ is supported in
$\{2nd+d, 2nd+d+1\}$; that is, $P_{3n+2}$ is
$\varDelta(3n+2)$-generated. This proves that $A$ is a bi-Koszul
algebra.

Conversely, due to Lemma \ref{le} and Proposition \ref{propp}, we
need only to show that the last isomorphism is true for each $n\geq
1$.

Since $E^{3n+2}(A)= \mathsf{E}^2(\Omega^{3n}(k))$ and
$\Omega^{3n}(k)$ is $2nd$-generated, we have the following minimal
resolution of $\Omega^{3n}(k)(-2nd)\in Gr_0(A)$:
$$
\cdots \rightarrow P_{3n+2}(-2nd)\rightarrow P_{3n+1}(-2nd)
\rightarrow P_{3n}(-2nd)\rightarrow \Omega^{3n}(k)(-2nd)\rightarrow
0.
$$
Denote $M=\Omega^{3n}(k)(-2nd)$. By Proposition \ref{prop3}, we know
that
$$
\mathsf{E}^2(M)\cong E^2(A)\mathsf{E}^0(M)\oplus \mathsf{E}^2_{d+1}(JM).
$$
Therefore,
$$
E^{3n+2}(A)\cong
E^2(A)E^{3n}(A)\oplus\mathsf{E}^2_{2nd+d+1}(J\Omega^{3n}(k)).
$$
\end{proof}

\begin{remark}
The obstruction $\mathsf{E}^2_{2nd+d+1}(J\Omega^{3n}(k))$ in
Theorem \ref{thm1} arises from the bigger degree in
$\varDelta(3n+2)$. It is vanished in the previous examples.
\end{remark}
\begin{defn}
We call a bi-Koszul algebra $A$ {\it strongly\/} if the obstruction
is vanished.
\end{defn}

Obviously, a strongly bi-Koszul algebra has nice homological
properties, for example, the Koszul dual $E(A)$ of such an algebra
is generated by $E^0(A), E^1(A), E^2(A)$ and $E^3(A)$. The following
gives a criterion for a bi-Koszul algebra to be strongly by
syzygies.
\begin{prop}\label{prop6}
Suppose that $A$ is a bi-Koszul algebra. Then $A$ is strongly if and
only if $\Omega^2(J\Omega^{3n}(k))$ and $\Omega^{3n+3}(k)$ are
generated in the same degrees for all $n\geq 1$.
\end{prop}
\begin{proof}
Assume $\Omega^{3n+3}(k)$ and $\Omega^2(J\Omega^{3n}(k))$ are
generated in the same degrees. Note that now $\Omega^{3n+3}(k)$ is
$(2nd+2d)$-generated, so does $\Omega^2(J\Omega^{3n}(k))$. That
means $\mathsf{E}^2_{2nd+d+1}(J\Omega^{3n}(k))=0$. So $A$ is
strongly bi-Koszul.

Conversely, from the proofs of Proposition \ref{prop3} and Theorem
\ref{thm1}, the assumption of
$\mathsf{E}^2_{2nd+d+1}(J\Omega^{3n}(k))=0$ concludes that
$\mathsf{E}^2(J\Omega^{3n}(k))$ is supported in $\{2nd+2d\}$.
Since $\Omega^{3n+3}(k)$ is $(2nd+2d)$-generated, we get
$\Omega^2(J\Omega^{3n}(k))$ and $\Omega^{3n+3}(k)$ are generated
in the same degrees.
\end{proof}

We will give another criterion for a bi-Koszul algebra being
strongly next section, and will provide a method of constructing a
strongly bi-Koszul algebra in the last section.

\vskip5mm
\section{Bi-Koszul Modules}\label{sect}

We mentioned a kind of modules whose degree distribution of
generators in the resolution obeys some resolution map $\varDelta$
in Proposition \ref{prop3}. In this section, we turn to study such
modules which are defined over the bi-Koszul algebras.

\begin{defn}Let $A$ be a bi-Koszul algebra. Assume $M\in Gr(A)$, we
call $M$ a {\it bi-Koszul module\/} if it has a minimal graded
projective resolution of the form
$$
\mathcal{Q}:\quad \cdots \rightarrow Q_n \stackrel{\partial_n}
\rightarrow \cdots \rightarrow Q_2\stackrel{\partial_2}\rightarrow
Q_1\stackrel{\partial_1}\rightarrow
Q_0\stackrel{\partial_0}\rightarrow M\rightarrow 0,
$$
in which each
$Q_n$ is generated in degrees $\varDelta(n)$.
\end{defn}

\begin{remark}
If $M$ is a bi-Koszul module, then $M\in Gr_0(A)$ and for each
$n\geq0$, $\Omega^{3n}(M)(-2nd)$ is also a bi-Koszul module.
\end{remark}

\begin{theorem}\label{thmm1}
Let $A$ be a bi-Koszul algebra, $M\in Gr_0(A)$. Then we have:
\begin{enumerate}
\item  if $\mathsf{E}(M)$ is generated in degree 0 as an $E(A)$-module,
then $M$ is a bi-Koszul module;
\item  if $M$ is a bi-Koszul module, then $\mathsf{E}(M)$ is
generated in degree 0 if and only if $\Omega^2(JM)(-2d)$ is also a
bi-Koszul module.
\end{enumerate}
\end{theorem}
\begin{proof}
(1) By the assumption, for each $n\geq 0$, $\mathsf{E}^n(M)=
E^n(A)\mathsf{E}^0(M)$, so
$$
\mathsf{E}^n(M)=E^n_{\varDelta(n)}(A)\mathsf{E}^0(M) =
\mathsf{E}^n_{\varDelta(n)}(M).
$$
Since $A$ is a bi-Koszul algebra, the result follows by Lemma
\ref{le}.

(2) When $M$ is a bi-Koszul module, the proof of Proposition
\ref{prop3} shows
\begin{eqnarray*}
\mathsf{E}(JM)&=&\mathsf{E}^0_1(JM)\oplus
\mathsf{E}^1_{d, d+1}(JM)\oplus \mathsf{E}^2_{d+1, 2d}(JM)\oplus
\mathsf{E}^3_{2d+1}(JM)\\ \quad & &\oplus
\mathsf{E}^4_{3d, 3d+1}(JM)\oplus \mathsf{E}^5_{3d+1, 4d}(JM)
\oplus \mathsf{E}^6_{4d+1}(JM)\oplus \cdots.
\end{eqnarray*}

By Lemma \ref{le2}, for all $n\geq 0$,
$\mathsf{E}^n(M)=E^n(A)\mathsf{E}^0(M)$ if and only if the induced
map $\mathsf{E}(M/JM)\rightarrow \mathsf{E}(M)$ is an epimorphism,
which is equivalent to
$\mathsf{E}^2_{d+1}(JM)=\mathsf{E}^5_{3d+1}(JM)=\cdots=0$. So we
have $\mathsf{E}^n(M)=E^n(A)\mathsf{E}^0(M)$ (for all $n\geq0$) if
and only if $\Omega^2(JM)(-2d)$ is also a bi-Koszul module.
\end{proof}

We discuss  some conditions under which $\Omega^2(JM)(-2d)$
becomes a bi-Koszul module.

Firstly, we consider a kind of special bi-Koszul modules.

\begin{defn}
Let $A$ be a bi-Koszul algebra, $M\in Gr_0(A)$ a bi-Koszul module.
We call $M$ {\it strongly\/} if, for all $n\geq 0$,
$J\Omega^{3n+2}(M)=J\Omega^{3n+2}(M/JM)\cap \Omega^{3n+2}(M)$.
\end{defn}

The following two lemmas are well-known.

\begin{lemma}\label{le4} For any graded algebra $A$ with Jacobson radical
$J$. Let $0\rightarrow K\rightarrow M \rightarrow N \rightarrow 0$
be a short exact sequence in $Gr(A)$ such that $JK=K\cap JM$, then
there exists a commutative diagram: {\small
$$
\begin{array}{ccccccccccc}&  & 0 &  & 0 &  & 0 &  &  \\
& & \downarrow  &  & \downarrow  & &  \downarrow & &  \\
0 & \to &\Omega(K) &  \to & \Omega(M)  & \to& \Omega(N) & \to & 0 \\
&  &  \downarrow  &  &  \downarrow  &  & \ \downarrow  &  &  \\
0& \to & Q &   \to & R & \to &L & \to & 0 \\
&  & \downarrow &  & \downarrow &  & \downarrow&  &  \\
0 & \to & K  &   \to & M & \to & N & \to & 0 \\
&  & \downarrow &  & \downarrow &  & \downarrow &  &  \\
&  & 0 &  & 0 &  & 0 &  &\\
\end{array}
$$
}
where $Q\rightarrow K$, $R\rightarrow M$ and $L\rightarrow N$ are
projective covers.\qed
\end{lemma}

\begin{lemma}\label{le5}
For any graded algebra $A$, let $0\rightarrow K\rightarrow M
\rightarrow N \rightarrow 0$ be a short exact sequence in $Gr(A)$.
If $K$, $M$ and $N$ are generated in one same degree, then $JK=K\cap
JM$.\qed
\end{lemma}

\begin{theorem}\label{thm6}
Let $A$ be a bi-Koszul algebra. Assume that $M\in Gr_0(A)$ is a
strongly bi-Koszul module. Then $\Omega^2(JM)(-2d)$ is a bi-Koszul
module. Hence, $\mathsf{E}(M)$ is generated in degree $0$ as an
$E(A)$-module.
\end{theorem}
\begin{proof}
Clearly, $M/JM$ is also a bi-Koszul module over $A$. For the short
exact sequence
$$
0\rightarrow JM\rightarrow M \rightarrow M/JM \rightarrow 0,
$$
we have the commutative diagram as follows: {\small
$$
\begin{array}{ccccccccccc}
& && & 0 &  & 0 &  & & &  \\
&  & &&\downarrow  &  & \downarrow  &   &  &  &  \\
&&0 & \to &\Omega^1(M) &  \to & \Omega^1(M/JM)  & \to& JM & \to & 0 \\
& && &  \downarrow  &  &  \downarrow  &  &  &  &  \\
&&& & Q_0 &   \to & Q_0 &  & &  &  \\
& && & \downarrow &  & \downarrow &  & &  &  \\
0 & \to & JM  &   \to & M & \to & M/JM & \to & 0&& \\
& && & \downarrow &  & \downarrow &  &  &  &  \\
& && & 0 &  & 0 .&  &  &  &\\
\end{array}
$$
}

Note that $M$ is strongly bi-Koszul, starting from the short exact
sequence
$$
0 \rightarrow\Omega^1(M) \rightarrow\Omega^1(M/JM) \rightarrow
JM \rightarrow 0,
$$
we may get projective covers sheaf by sheaf from the lemmas
\ref{le4} and \ref{le5}. Hence, for all $n \geq 1$,
$$
0\rightarrow\Omega^n(M)\rightarrow\Omega^n(M/JM)\rightarrow\Omega^{n-1}(
JM)\rightarrow 0
$$
are exact, where $\Omega^n(M)$, $ \Omega^n(M/JM)$ and
$\Omega^{n-1}(JM)$  are all generated in the same degrees. So
$\Omega^2(JM)(-2d)$ is a bi-Koszul module by comparing with
degrees.
\end{proof}


The following result is evident from the proof of Theorem
\ref{thm6}.

\begin{corollary}\label{cor3}
Under the condition of Theorem \ref{thm6}, we have, for all $n \geq 1$,
$$
0\rightarrow \mathsf{E}^{n-1}(JM) \rightarrow
\mathsf{E}^n(M/JM)\rightarrow \mathsf{E}^n( M)\rightarrow 0
$$
are exact, where $\mathsf{E}^{n-1}(JM)$, $\mathsf{E}^n(M/JM)$ and
$\mathsf{E}^n(M)$ are all supported in the same degrees.
\end{corollary}

The following proposition gives another criterion for a bi-Koszul
algebra to be strongly by the notion of bi-Koszul modules.
\begin{prop}
Let $A$ be a bi-Koszul algebra. Assume, for all $n\geq 1$,
$\Omega^{3n}(k)(-2nd)$ are strongly bi-Koszul modules. Then $A$ is a
strongly bi-Koszul algebra.
\end{prop}
\begin{proof}
By Proposition \ref{propp}, we know that
$$
E^{3n}(A)=\underbrace{E^3(A)E^3(A)\cdots E^3(A)}_n
$$
and
$$
E^{3n+1}(A) = E^1(A)E^{3n}(A).
$$
Denote $M=\Omega^{3n}(k)(-2nd)$, which is a strongly bi-Koszul
module. So by Theorem \ref{thm6}, we get
$$
\mathsf{E}^2(M)=E^2(A)\mathsf{E}^0(M);
$$
that is,
$$
E^{3n+2}(A)=E^2(A)E^{3n}(A).
$$
Thus, we complete the proof.
\end{proof}

We end this section with a result about the relations between
bi-Koszul modules and Koszul modules.

Denote $\mathsf{E}^{[0]}(M)=\oplus_{n\geq0}\mathsf{E}^{3n}(M)$, and
$E^{[0]}(A)=\oplus_{n\geq0}E^{3n}(A)$. Clearly,
$\mathsf{E}^{[0]}(M)$ is a left $E^{[0]}(A)$-module. Denote $M_0:=M$
and $ M_n:=\Omega^2(JM_{n-1})(-2d)$ for $n\geq 1$.
\begin{prop}
Let $A$ be a bi-Koszul algebra. Assume that all $M_n \;(n\geq 0)$
are strongly bi-Koszul modules. Then $\mathsf{E}^{[0]}(M)$ is a
Koszul module over $E^{[0]}(A)$.
\end{prop}
\begin{proof}
By Corollary \ref{cor3}, there exist short exact sequences:
$$
0\rightarrow \mathsf{E}^{3n-1}(JM) \rightarrow \mathsf{E}^{3n}(M/JM)\rightarrow
\mathsf{E}^{3n}(M)\rightarrow 0
$$
for all $n\geq 0$ (we have denoted $\mathsf{E}^{-1}(JM)=0$). Hence
$$
0\rightarrow \bigoplus_{n\geq 1}\mathsf{E}^{3n-3}(\Omega^2(JM))
\rightarrow \bigoplus_{n\geq 0}\mathsf{E}^{3n}(M/JM)\rightarrow
\bigoplus_{n\geq 0}\mathsf{E}^{3n}(M)\rightarrow0;
$$
that is,
$$
0\rightarrow
\bigoplus_{n\geq 0}\mathsf{E}^{3n}(\Omega^2(JM)(-2d))(-1)
\rightarrow \bigoplus_{n\geq 0}\mathsf{E}^{3n}(M/JM)\rightarrow
\bigoplus_{n\geq 0}\mathsf{E}^{3n}(M)\rightarrow 0.
$$
Since $\Omega^2(JM)(-2d)$ is a strongly bi-Koszul module,
$\bigoplus_{n\geq 0}\mathsf{E}^{3n}(\Omega^2(JM)(-2d))(-1)$ $=
\Omega^1(\mathsf{E}^{[0]}(M))$ is 1-generated by Theorem
\ref{thm6}. Continuing this fashion, we see that
$\mathsf{E}^{[0]}(M)$ has a linear resolution. The result now
follows.
\end{proof}

\vskip5mm
\section{Decomposition of Resolutions}\label{secti}

In this section, we discuss another topic that whether a minimal
resolution of a bi-Koszul module can be decomposed into two pure
resolutions. We end the article by providing a method of
constructing the strongly bi-Koszul algebras.

From now on, we assume that the graded algebra $A$ is noetherian and
connected, here ``connected'' means $k$ itself is a field. All
modules are considered finitely generated graded free A-modules with
their homogeneous bases. Denote $|x|$ the degree of a homogeneous
element $x\in M$ for $M\in gr(A)$.

Using the notation recalled in the preliminary section, we write
graded free $A$-modules as $$M= A(-m_1 , \cdots ,-m_s),\;\;\; N=
A(-n_1 , \cdots ,-n_t)$$ with all $m_i, n_j \in \mathbb{N}$
satisfying $m_1\leq \cdots \leq m_s$, and $n_1\leq \cdots \leq n_t$.
We choose homogeneous $A$-bases $\zeta$ and $\xi$ for $M$ and $N$
respectively, where
$$
\zeta= \left(\begin{array}{c} x_1\\ \vdots \\ x_s\\ \end{array}
\right),\;\; \xi=\left(\begin{array}{c} y_1 \\ \vdots \\ y_t \\
\end{array}\right),
$$
with $|x_i|= m_i$ and $|y_j|= n_j$.

For $f\in \Hom_{Gr(A)}(M, N)$, set $F=(a_{ij})\in
\mathbb{M}_{s\times t}(A)$ and say that $F$ is a {\it matrix
representation\/} of $f$ (with respect to the bases $(\zeta, \xi)$)
in case $f(\zeta)=F\xi$. In this case, we simply say that $(f, F)$
forms a {\it match}.

\begin{remark}
\begin{enumerate}
\item If $f\ne 0$, one can compute out that $|a_{ij}|=
m_i- n_j$ if $m_i\geq n_j$; and $a_{ij}=0$ if $m_i<n_j$.
\item There exists a nonzero entry in each row of
$F$, all nonzero entries in $F$ are homogeneous.
\end{enumerate}
\end{remark}

\begin{lemma}\label{le7}
Let $(f, F)$ be a match with respect to $(\zeta, \xi)$, assume
$F=\left(\begin{array}{cc} F' & 0\\ U &F''\\ \end{array} \right)$
with both $F'$ and $F''$ being nonzero. Then there exist a
commutative diagram:
$$
\begin{array}{ccclcrccc}
0 & \to & M' &  \to & M & \to& M''& \to & 0 \\
&  &  \downarrow f' &  &  \downarrow f &  & \ \downarrow f'' &  &  \\
0 & \to & N' &   \to & N & \to & N'' & \to & 0
\end{array}
$$
such that each row is split exact, and all $(f', F')$, $(f'', F'')$,
$(\lambda, U)$ are matches, where $\lambda\in\Hom_{Gr(A)}(M'', N')$.
\end{lemma}

\begin{proof} Let $F'\in \mathbb{M}_{s'\times t'}(A)$.
Set $\zeta=\left(\begin{array}{c} \zeta' \\ \zeta'' \\
\end{array}\right) $where $\zeta'$ has $s'$ rows, and $\xi=\left(\begin{array}{c}
\xi' \\ \xi'' \\\end{array} \right)$ where $\xi'$ has $t'$ rows. For
$x\in M$ with $x=X\zeta$, write $X=(X', X'')$ where $X'$ has  $s'$
columns. Since $f(x)=f(X\zeta)= XF\xi$, we have
$$
f(X'\zeta'+X''\zeta'')= X'F'\xi'+X''U\xi'+X''F''\xi''.
$$
Set $M'$, $M''$, $N'$, $N''$ to be free modules generated by
$\zeta'$, $\zeta''$, $\xi'$, $\xi''$, respectively, and
$f'=f|_{M'}$, $f''=P_{N''}\circ f|_{M''}$, $\lambda=P_{N'}\circ
f|_{M''}$, here $P|_{N'}: N\to N'$ and $P|_{N''}: N\to N''$ are
projective maps. Then the result follows.
\end{proof}

\begin{defn} Let notation be assumed as in the lemma above, suppose that for any
$X''$, there exists  $X'$ such that $X''U=X'F'$, then we call the
matrix $F$ {\it admissible}.
\end{defn}
\begin{remark}
\begin{enumerate} \item $F=\left(\begin{array}{cc} F' & 0\\ 0 &F''\\
\end{array}\right)$ is admissible.
\item If a match $(f, F)$ with $F$ admissible, then
$\mbox{Im}\lambda\subseteq \mbox{Im}f'$ from the proof of Lemma 4.1.
\end{enumerate}
\end{remark}

\begin{lemma}\label{le8} Assume that $M \stackrel{f}\rightarrow N
\stackrel{g}\rightarrow P$ is an  exact sequence, and $(f, F)$, $(g,
G)$ are matches
such that both $F =\left(\begin{array}{cc} F' & 0\\ U &F''\\
\end{array}
\right)$ and  $G=\left(\begin{array}{cc} G' & 0\\ V &G''\\
\end{array} \right)$ are admissible. Then, we have the following
commutative diagram:
{\small
$$
\begin{array}{cccccrccc}
&  & 0 &  & 0 &  & 0 &  &  \\
&  & \downarrow  &  & \downarrow  &   &  \downarrow &  &  \\
&  & M' & \stackrel{f'}\to & N' &\stackrel{g'} \to& P'&   &   \\
&  &  \downarrow  &  &  \downarrow &  &   \downarrow  &  &  \\
&  & M & \stackrel{f}\to &N & \stackrel{g}\to & P &  &   \\
&  & \downarrow &  & \downarrow &  & \downarrow&  &  \\
&  &M'' & \stackrel{f''}\to &  N'' & \stackrel{g''}\to & P'' &  &   \\
&  & \downarrow &  & \downarrow &  & \downarrow &  &  \\
&  & 0 &  & 0 &  & 0 &  &
\end{array}
$$
}
such that all rows are exact, and all columns are split exact.
\end{lemma}

\begin{proof}
By Lemma \ref{le7} and the assumption, we obtain $M'$, $N'$, $ P' $,
$M''$, $ N''$ and $ P''$ such that three columns are split exact and
the diagram is commutative. Moreover, there exist $\lambda_1\in
\Hom_{Gr(A)}(M'', N')$, $\lambda_2\in \Hom_{Gr(A)}(N'', P')$ such
that
$$
f(x)=f'(x')+\lambda_1(x'')+f''(x''), \quad g(y)=g'(y')+\lambda_2(y'')+g''(y'')
$$
satisfying $\mbox{Im}\lambda_1\subseteq \mbox{Im}f'$ and
$\mbox{Im}\lambda_2\subseteq \mbox{Im}g'$.

Now we prove the first row is exact.

It is clear that $\ker g'=\ker g\cap N'$. We show that
$\mbox{Im}f'=\mbox{Im}f \cap N'$. It is obvious that
$\mbox{Im}f'\subseteq\mbox{Im}f \cap N'$. For any $y\in \mbox{Im}f
\cap N'$, there exists $x=x'+x''\in N$ such that
$y=f(x)=f'(x')+\lambda_1(x'')+f''(x'')$ with
$f'(x')+\lambda_1(x'')\in N'$ and $f''(x'')\in N''$. So
$f(x)=f'(x')+\lambda_1(x'')$ and $f''(x'')=0.$ Since
$\mbox{Im}\lambda_1\subseteq \mbox{Im}f'$, there exists $z\in M'$
such that $\lambda_1(x'')=f'(z)$. Thus, $y=f(x)=f'(x'+z)$.
Therefore, $y\in\mbox{Im}f'$. So we have $\ker g'=\ker g\cap
N'=\mbox{Im}f \cap N'=\mbox{Im}f'.$

Next, we show that the third row is exact.

We prove $\mbox{Im}f''=\mbox{Im}f \cap N''$ firstly. In fact, $
\mbox{Im}f \cap N''\subseteq\mbox{Im}f''$ is obvious. Conversely,
for any $y\in \mbox{Im}f''$, there exists $x\in M''$ such that
$f''(x)=y$, we get $f(x)=\lambda_1(x)+f''(x)$. Since
$\mbox{Im}\lambda_1\subseteq \mbox{Im}f'$, there exists $z\in M'$
such that $f'(z)=\lambda_1(x)$. So $f(x-z)=
\lambda_1(x)+f''(x)-f'(z)=f''(x)=y$. Therefore, $y\in \mbox{Im}f
\cap N''.$ This proves $\mbox{Im}f''=\mbox{Im}f \cap N''$.

On the other hand, we show that $\ker g''=\ker g\cap N''$. It is
obvious that $\ker g\cap N''\subseteq \ker g''$.  For any $y\in \ker
g''$, we have $y \in N''$ and $g''(y)=0$. Since
$\mbox{Im}\lambda_2\subseteq \mbox{Im}g'$, there exists $z\in N'$
such that $g'(z)=\lambda_2(y)$. We get
$g(z-y)=g'(z)-\lambda_2(y)-g''(y)=0$. Since $\ker g=\mbox{Im}f$,
there exists $v=v'+v''\in M$ such that
$f(v)=f'(v')+\lambda_1(v'')+f''(v'')=z-y$. Clearly,
$f'(v')+\lambda_1(v'')=z$ and $f''(v'')=-y$. Since
$\mbox{Im}\lambda_1\subseteq \mbox{Im}f'$, there exists $w\in M'$
such that $f'(w)=\lambda_1(v'')$. Thus, $f'(v'+w)=z$. We get
$\lambda_2(y)=g'(z)=g'f'(v'+w)=0$ by the exactness of the first row.
So $g(y)=\lambda_2(y)+g''(y)=0$. Therefore, $y\in \ker g\cap N''$.

Hence, we get $\ker g''=\ker g\cap N''=\mbox{Im}f \cap
N''=\mbox{Im}f''$.
\end{proof}

With preparations above, we turn our attention to a kind of
bi-Koszul modules.

Let $M$ be a bi-Koszul module over $A$ with the minimal free
resolution:
$$
\mathcal{Q}: \quad \cdots\rightarrow Q_{3n+2}
\stackrel{\partial_{3n+2}}\rightarrow
Q_{3n+1}\stackrel{\partial_{3n+1}}\rightarrow Q_{3n}
\stackrel{\partial_{3n}}\rightarrow\cdots \rightarrow
Q_2\stackrel{\partial_2}\rightarrow Q_1
\stackrel{\partial_1}\rightarrow Q_0\stackrel{\varepsilon}
\rightarrow M \rightarrow 0
$$
where  $Q_{3n+2}=A(-2nd-d)^{p_{3n+2}}\oplus
A(-2nd-d-1)^{q_{3n+2}}$.

We say that $M$ is {\it decomposable\/} in case, for each $m\geq 1$,
the matrix representation $F_m$ of $\partial_m$ with the form
$\left(\begin{array}{cc} F_m' & 0\\U_m &F_m''\\\end{array} \right)$
satisfying
\begin{enumerate}
\item $F_m$ is admissible (if the resolution is ending at $Q_m$ with $m\ne 3n+2$,
$F_m=(F_m'\; 0)$ is permitted),
\item $F_{m+1}F_m$ is consistent according to the block multiplication, and
\item the number of rows of $F'_{3n+2}$ is equal to $p_{3n+2}$.
\end{enumerate}

\begin{theorem}\label{thm7}
Let $A$ be a connected noetherian graded algebra, and
$M$ a decomposable bi-Koszul module defined as above. Then the
resolution $\mathcal{Q}$ can be decomposed into a direct sum of two
pure resolutions. Moreover, there exists a short exact sequence
$0\rightarrow M'\rightarrow M\rightarrow M''\rightarrow0$ such that
$M'$ and $ M''$ are $ \delta$-Koszul modules.
\end{theorem}
\begin{proof}
With the assumption, applying the lemmas \ref{le7} and \ref{le8} to
each exact sequence $Q_{m+1}\stackrel{\partial_{m+1}}\rightarrow Q_m
\stackrel{\partial_m}\rightarrow Q_{m-1}$ produces two exact
sequences
$$
\cdots\rightarrow Q'_m\stackrel{\partial'_m} \rightarrow \cdots
\rightarrow Q'_2\stackrel{\partial'_2}\rightarrow Q'_1
\stackrel{\partial'_1}\rightarrow Q'_0,
$$
and
$$
\cdots\rightarrow Q''_m\stackrel{\partial''_m} \rightarrow\cdots
\rightarrow Q''_2\stackrel{\partial''_2}\rightarrow Q''_1
\stackrel{\partial''_1}\rightarrow Q''_0,
$$
where both $Q'_m$ and $Q''_m$ are pure, moreover, $Q_m=Q'_m\oplus
Q''_m$. Denote that
$$
M'=\varepsilon(Q'_0),\;\;\; M''=\pi\epsilon(Q''_0),
$$
where $\pi$ is the canonical surjective morphism from $M$ to $M/M'$.
We have the commutative diagram:
$$
\begin{array}{ccclcrccc}
&  & 0 &  & 0 &  & 0 &  &   \\
&  & \downarrow  &  & \downarrow  &   &  \downarrow &  &  \\
&  & Q'_1 & \stackrel{\partial'_1}\to & Q'_0 &\stackrel{\epsilon'} \to& M'& \to  &  0 \\
&  &  \downarrow  &  &  \downarrow &  &   \downarrow  &  &  \\
&  & Q_1 & \stackrel{\partial_1}\to &Q_0 & \stackrel{\varepsilon}\to & M & \to &  0 \\
&  & \downarrow &  & \downarrow &  & \downarrow&  &  \\
&  & Q''_1 & \stackrel{\partial''_1}\to &  Q''_0 & \stackrel{\varepsilon''}\to & M'' & \to & 0  \\
&  & \downarrow &  & \downarrow &  & \downarrow &  &  \\
&  & 0 &  & 0 &  & 0 &  &
\end{array}
$$
with $\varepsilon'=\varepsilon\mid _{Q'_0}$ and
$\varepsilon''=\pi\varepsilon\mid _{Q''_0}$. It is clear that
$\varepsilon''$ is an epimorphism. The first and third rows are
exact by the similar proof of Lemma \ref{le8}.

If the resolution $\mathcal{Q}$ is ending at $Q_m$ with $m\ne 3n+2$,
and $F_m$ has the form $(F_m'\; 0)$, we can get the similar proof by
taking $Q_m'=Q_m,\, Q_m''=0$.
\end{proof}

\begin{remark}\label{remk}
We may describe $M'$ and $M''$ above more precisely
by their resolutions:
$$
\mathcal{Q'}: \cdots\rightarrow Q'_m\stackrel{\partial'_m}
\rightarrow \cdots \rightarrow
Q'_2\stackrel{\partial'_2}\rightarrow Q'_1
\stackrel{\partial'_1}\rightarrow Q'_0\stackrel{\varepsilon'}
\rightarrow M' \rightarrow 0
$$
and
$$
\mathcal{Q''}: \cdots\rightarrow Q''_m \stackrel{\partial''_m}
\rightarrow \cdots \rightarrow
Q''_2\stackrel{\partial''_2}\rightarrow Q''_1
\stackrel{\partial''_1}\rightarrow Q''_0\stackrel{\varepsilon''}
\rightarrow M'' \rightarrow 0
$$
where $Q'_m$ is generated in degree
$$
\delta'(m)=\left\{\begin{array}{llll} 2nd, & \mbox{if $m = 3n$,}\\
2nd+1,  &  \mbox{if $m = 3n+1$,}\\2nd+d,   & \mbox{if $m = 3n+2$,}
\end{array}\right.
$$
and $Q''_m$ is generated in degree
$$
\delta''(m)=\left\{\begin{array}{llll}2nd,  &  \mbox{if $m = 3n$,}\\
2nd+1,  &  \mbox{if $m = 3n+1$,}\\2nd+d+1,   & \mbox{if $m = 3n+2$.}
\end{array}
\right.
$$

Both of them are one kind of $\delta$-Koszul modules introduced in
\cite{GM}, but here the notion is defined over an arbitrary graded
algebra.
\end{remark}

\begin{corollary} In the case of Theorem \ref{thm7}, we get,
for all $m\geq 1$, $\mathsf{E}^m(M)=\mathsf{E}^m(M')\oplus
\mathsf{E}^m(M'')$, where both $\mathsf{E}^m(M')$ and
$\mathsf{E}^m(M'')$ are pure.
\end{corollary}

\begin{exa}
Set
$$
A=\frac{k\langle x, y, z, w\rangle}{(yx, z^2y, wz)}.
$$
The trivial module $_Ak$ has the minimal resolution:
$$
0\rightarrow A(-5)  \rightarrow A(-4)^2 \rightarrow
A(-2)^2\oplus A(-3) \rightarrow A(-1)^4 \rightarrow A
\rightarrow k \rightarrow 0,
$$
so $A$ is a bi-Koszul algebra. Consider $M:=\dfrac{A\oplus
A}{\big((x,0), (0,y)\big)}$, it is a bi-Koszul module because it has
the minimal resolution:
$$
0 \rightarrow A(-5)  \stackrel{M_{4}}\rightarrow A(-4)^2
\stackrel{M_3}\rightarrow A(-2)\oplus A(-3)
\stackrel{M_2}\rightarrow A(-1)^2 \stackrel{M_1}\rightarrow
A^2 \stackrel{\varepsilon} \rightarrow M \rightarrow 0.
$$
where $M_4=\left(\begin{array}{cc} w & 0 \\\end{array} \right)$,
$M_3=\left(\begin{array}{cc} z^2 & 0 \\ 0 & w \\\end{array}
\right)$, $M_2=\left(\begin{array}{cc} y & 0 \\ 0 & z^2\\
\end{array}\right)$ and $M_1=\left(\begin{array}{cc} x & 0\\
0 & y \\\end{array}\right)$. So by  Theorem \ref{thm7} we have the
minimal resolutions:
$$
0 \rightarrow A(-5)  \rightarrow A(-4)  \rightarrow A(-2)
\rightarrow A(-1)\rightarrow A \rightarrow A/(x) \rightarrow 0,
$$
and
$$
0 \rightarrow A(-4)  \rightarrow A(-3)  \rightarrow A(-1)
\rightarrow A \rightarrow A/(y) \rightarrow 0.
$$
\end{exa}

\medskip

We end this article by constructing a strongly bi-Koszul algebra
based on the notion of free products and direct sums,  as well as
the existence of $\delta$-Koszul algebra mentioned above.

For the details of the free product $A\sqcup A'$ and direct sum
$A\sqcap A'$, we refer to (\cite{PP}, Chapter 3), from which the
following results will be used in our construction.
\begin{enumerate}
\item $E(A\sqcup A')=E(A)\sqcap E(A')$,
\item $(E(A)\sqcap E(A'))_i=\left\{\begin{array}{llll}
k, & \mbox{if $i=0$,}\\
E^i(A)\oplus E^i(A'), &  \mbox{if $i>0$,}
\end{array}\right.$
\item $E^{>0}(A)\cdot E^{>0}(A')=E^{>0}(A')\cdot E^{>0}(A)=0$.
\end{enumerate}

Now suppose that $A$ and $A'$ are two $\delta$-Koszul algebras, the
trivial modules $_Ak$ and $_{A'}k$ have their minimal resolutions
whose degree distributions obey the resolution maps $\delta'$ and
$\delta''$ mentioned in Remark \ref{remk}. We know that
$E^i(A)E^{3n}(A)=E^{i+3n}(A)$ and $E^i(A')E^{3n}(A')=E^{i+3n}(A')$
for $i=0, 1, 2$ and $n\geq 0$ by Lemma \ref{gmmz} .

Denote $B=A\sqcup A'$, we have $E^{3n+i}(B) =  E^{3n+i}(A)\oplus
E^{3n+i}(A')$. On the other hand, $E^i(B)E^{3n}(B)=(E^i(A)\oplus
E^i(A'))(E^{3n}(A)\oplus E^{3n}(A'))=E^i(A)E^{3n}(A)\oplus
E^i(A')E^{3n}(A')$ for $i=0, 1, 2$ and $n\geq 0$. Hence, $B$ is a
bi-Koszul algebra and $E^i(B)E^{3n}(B)=E^{3n+i}(B)$ for $i=0, 1, 2$
and $n\geq 0$. So $B$  is a strongly bi-Koszul algebra.

The above shows that one may get a strongly bi-Koszul algebra  by
free products of some graded algebras with pure resolutions.
Unfortunately, the constructing can not product any non-strongly
bi-Koszul algebra.

\medskip
We leave the following question.

\textsf{Question.} Is there a bi-Koszul algebra that is not
strongly? Or equivalently, must the bi-Koszul algebras be strongly?

\vskip5mm

\textit{Acknowledgments.} The authors thank Roland Berger and Nicole
Snashall for their useful comments and reminding of the papers
\cite{BBK, GS} in the development of Koszul algebras.

\vskip1cm
\bibliographystyle{amsplain}

\end{document}